\documentclass{amsart}

\usepackage{amssymb, amsmath}
\usepackage{graphicx}

\usepackage[nameinlink,capitalize]{cleveref} 
\crefname{equation}{}{}  
\crefname{figure}{Figure}{Figures}

\theoremstyle{definition}
\newtheorem{thm}{Theorem}

\newtheorem{lemma}[thm]{Lemma}
\newtheorem{cor}[thm]{Corollary}
\newtheorem{definition}[thm]{Definition}

\title[Combinatorial causal set d'Alembertian]{Combinatorial interpretation of the coefficients of the causal set d'Alembertian}
\author{Karen Yeats}
\address{Department of Combinatorics and Optimization, University of Waterloo, Waterloo, ON, Canada, kayeats@uwaterloo.ca}
\thanks{The author is supported by an NSERC Discovery grant and by the Canada Research Chairs program.  The author is grateful to Max Planck Institute for Mathematics in Bonn for its hospitality and financial support. They also thank Rafael, Sorkin, Fay Dowker, Stav Zalel and Nick Olson-Harris for useful conversations.}

\begin{document}
\maketitle

\begin{abstract}
The causal set theory d'Alembertian has rational coefficients for which alternating expressions are  known.  Here, a combinatorial interpretation of these numbers is given.
\end{abstract}

\section{Introduction}

Causal set theory is an approach to quantum gravity where the causal relation between spacetime points and discreteness of spacetime are the two fundamental assumptions upon which one attempts to build everything else.  Consequently, in this approach, one works with locally finite posets (also known as \emph{causal sets}) as the spacetime where the partial order is the causal order in the spacetime, \cite{BLMS}. For some reviews see \cite{Scstreview, Yeverything, DSreview}.

To quickly set poset notation and language, a \emph{poset} $P$ is a set (by standard abuse of notation also called $P$), with a reflexive, antisymmetric, and
transitive relation which we'll write $\leq$.  For $a, b\in P$ we write $a<b$ if $a\leq b$ and $a\neq b$.  We say $b$ \emph{covers} $a$ if $a<b$ but there is no $c\in P$ such that $a<c<b$.  The \emph{interval} from $a$ to $b$ is $[a,b]=\{c\in P: a\leq c\leq b\}$.  Our intervals are by definition closed intervals, sometimes in causal set theory open intervals (not including $a$ and $b$) are used instead; it is straightforward to convert between the different conventions.  A poset is \emph{locally finite} if all its intervals are finite sets.  For more on posets see \cite{Trotter}.

Returning to the causal set context, while one ultimately would like to begin with causal sets alone, without reference to an underlying continuous spacetime, it is often useful to start with a Lorentzian manifold and choose points randomly on it according to a Poisson distribution, letting those points with the induced causal order be our causal set.  Here the probability to find $m$ poset elements in a volume $V$ is $(\rho V)^m e^{-\rho V}/m!$, where $\rho$ is a density parameter.  Instead of working with $\rho$ it is often convenient to work in terms of a length parameter, called the \emph{discreteness scale}, $\ell$ defined by $\rho = \ell^{-d}$.  Generating a causal set in this manner is called \emph{sprinkling}, see Section 3 of \cite{Scstreview}. 

One of the many operators which one has in the continuum case, which one would like an analogue of for causal sets, is the d'Alembertian or box operator, which in 4-dimensional Minkowski space is
\[
\square = \frac{1}{c^2}\frac{\partial^2}{\partial t^2} - \frac{\partial^2}{\partial x^2} - \frac{\partial^2}{\partial y^2} - \frac{\partial^2}{\partial z^2}.
\]
Generally speaking, when moving to the discrete situation, one would expect differential operators, as in the classical d'Alembertian, to become difference operators.  However, causal sets are only \emph{locally} finite posets, so we can have infinitely many nearest neighbours of any given point, that is for $a$ in the poset there can be infinitely many elements covering $a$ and there can be infinitely many elements that $a$ covers.  Consequently, there's no reason to expect that a calculation involving nearest neighbours would converge in an infinite but locally finite poset, nor, if we take finite approximations is there a reason to expect a calculation involving nearest neighbours in the approximations to converge in any reasonable limit.

However, Sorkin in \cite{Sdalembertian} observed that a particular alternating linear combination has the property that it approximates the d'Alembertian of 2-dimensional Minkowski space in the following sense.  Let $C$ be a causal set and let $\phi$ be a scalar field on $C$, which for us just means a function from $C$ to $\mathbb{C}$.  Define
\[
  B^{(2)}\phi(x) = \frac{1}{\ell^2}\left( -2\phi(x) + 4 \left(\sum_{y\in L_1(x)}\phi(y) - 2\sum_{y \in L_2(x)}\phi(y) + \sum_{y\in L_3(x)}\phi(y)\right)\right)
\]
where $L_i(x) = \{y\in C: |[y,x]| = i+1\}$.  Then letting $\langle B\phi(x)\rangle$ be the expectation with respect to Poisson sprinkling on 2-dimensional Minkowski space, Sorkin proves that
\[
\lim_{\rho\rightarrow \infty}\frac{1}{\sqrt{\rho}}\langle B\phi(x)\rangle = \square \phi.
\]
This result was generalized by Benincasa and Dowker to 4 dimensions \cite{BDdalembertian}, by Dowker and Glaser to general dimension  \cite{DGdalembertian} but without an explicit expression past $d=7$, and then finally Glaser in \cite{Gdalembertian} gave the following general form for the $d$-dimensional causal set d'Alembertian:
\[
B^{(d)}\phi(x) = \frac{1}{\ell^2} \left(\alpha_d \phi(x) + \beta_d \sum_{i=1}^{\lfloor \frac{d}{2} \rfloor + 2 }C_i^{(d)}\sum_{y\in L_i}\phi(y)\right)
\]
where
\begin{align*}
  \beta_d & = \begin{cases}
  \displaystyle\frac{2 \Gamma\left(\frac{d}{2}+2\right)\Gamma\left(\frac{d}{2}+1\right)}{\Gamma\left(\frac{2}{d}\right)\Gamma(d)}c^{\frac{2}{d}}_d & \text{for even $d$} \\
  \displaystyle\frac{d+1}{2^{d-1}\Gamma\left(\frac{2}{d}+1\right)}c^{\frac{2}{d}}_d & \text{for odd $d$}\end{cases} \\
  \alpha_d & = \begin{cases}
    \displaystyle\frac{-2c^{\frac{2}{d}}_d}{\Gamma\left(\frac{d+2}{d}\right)} & \text{for even $d$} \\
    \displaystyle\frac{-c^{\frac{2}{d}}_d}{\Gamma\left(\frac{d+2}{d}\right)} & \text{for odd $d$}\end{cases} \\
  C_i^{(d)} & = \begin{cases}
    \displaystyle\sum_{k=0}^{i-1}\binom{i-1}{k}(-1)^k \frac{\Gamma\left(\frac{d}{2}(k+1) + 2\right)}{\Gamma\left(\frac{d}{2}+2\right)\Gamma\left(1+\frac{dk}{2}\right)} & \text{for even $d$}\\
    \displaystyle\sum_{k=0}^{i-1}\binom{i-1}{k}(-1)^k \frac{\Gamma\left(\frac{d}{2}(k+1)+\frac{3}{2}\right)}{\Gamma\left(\frac{d+3}{2}\right)\Gamma\left(1+\frac{dk}{2}\right)} & \text{for odd $d$}
    \end{cases}
\end{align*}
and $c_d = S_{d-2}/(d(d-1)2^{d/2-1})$ with $S_{d-2}$ the volume of the $d-2$ dimensional unit sphere.

These $B^{(d)}$ when applied to sprinklings on $d$-dimensional Minkowski space likewise limit to the $d$-dimensional classical d'Alembertian \cite{Gdalembertian}.

The causal set d'Alembertian has generated substantial interest in the literature, including generalizations where the upper limit of the sum exceeds $\frac{d}{2}$ resulting in non-uniqueness \cite{ASSdalembertian, Bdalembertian}, results refining our understanding of the behaviour of the continuum limit \cite{BBDcntlimit, Jdiscreteness}, generalizing to higher curvature operators \cite{dBEP} as well as using it for building quantum field theories on causal sets \cite{GDFcstaqft} and contrasts with other ways of building partial derivatives and similar operators on causal sets \cite{Spdcst}.

These same coefficients also appear in the Benincasa-Dowker-Glaser (BDG) action in causal set theory \cite{BDdalembertian}.  This is because on a curved spacetime Benincasa, Dowker, and Glaser proved that $\lim_{\rho \rightarrow \infty} \langle B^{(d)}\phi(x)\rangle = \square\phi -\frac{1}{2}R(x)\phi(x)$  \cite{BDdalembertian, DGdalembertian, Gdalembertian} where $R(x)$ is the Ricci scalar of the continuous spacetime.  This gives a candidate for an analogue of the Ricci scalar for a causal set.  Applying this to the constant field and summing over all elements of $C$ then gives the analogue of the Einstein-Hilbert action for the causal set, which is the BDG action:
\[
  S^{(d)}(C) = -\alpha_d \ell^{d-2} \left( N + \frac{\beta_d}{\alpha_d}\sum_{i=1}^{\lfloor \frac{d}{2} \rfloor + 2}C_i^{(d)}N_i\right)
\]
where $N=|C|$ and $N_i$ is the number of intervals of size $i+1$ in $C$.\footnote{This differs from \cite{Gdalembertian} by an overall factor independent of $d$, involving the rationalized Planck length, which is not important for our purposes.}  The BDG action is particularly important in causal set theory; its boundary and fluctuation effects are interesting and may relate to the Everpresent $\Lambda$ \cite{Dboundary, MWboundary, DLLJboundary, MYZlambda, DNYlambda}, it allows us to define a path integral from whence we can investigate which posets dominate under this action \cite{LCsuppress, MSSsuppress, CCSsuppress, CCSsuppress2}, and is useful for random generation \cite{Gcstmc}.

\medskip

Glaser comments at the end of Section 2 of \cite{Gdalembertian}, that the expressions for the $C_i^{(d)}$, while fast to compute, remain aesthetically unsatisfying since they could not simplify the sum.  From an enumerative perspective, these sums cry out for a combinatorial explanation, one where the alternation can be realized as explicit cancellations. This is what the present paper accomplishes.  Hopefully the reader will find this aesthetically satisfying, even without non-alternating expressions for the counts.  Furthermore, this combinatorial interpretation of the coefficients is physically nice both in explaining the cancellations between terms and in that it does not need a different form in even and odd dimensions.

\section{A class of noncrossing partial chord diagrams}

We take a two phase approach to interpreting the $C_i^{(d)}$.  The first phase is to define objects which give an interpretation for the ratio of gamma factors that appears in the sum.  Furthermore, we want a unified interpretation that works for both the odd and even dimensions.  The second phase is to modify these objects (in a way that's counted by the explicit binomial coefficients in the sum) so that the alternating sum can be realized by explicit cancellations, thus giving a class of objects which are counted by the $C_i^{(d)}$ directly.

The reader should note that the objects we define are certainly not the only objects which would have this property, but these ones seem reasonable and do the job.

\subsection{Phase 1}

Our main objects of interest are a particular class of coloured partial non-crossing rooted chord diagrams.

\begin{definition}
  A \emph{partial rooted chord diagram} with $n$ chords on $m$ points is a set partition of $\{1, 2, \ldots, m\}$ where every part has size $1$ or $2$ and there are $n$ parts of size $2$.  The parts of size $2$ are called \emph{chords}.  A \emph{rooted chord diagram} with $n$ chords is a partial rooted chord diagram with $n$ chords on $2n$ points, or equivalently where no parts have size $1$.
\end{definition}

When it is useful to emphasize a rooted chord diagram is not a partial diagram we'll call it a \emph{complete} rooted chord diagram.

\begin{definition}
  Two chords $\{a,b\}$ and $\{c,d\}$ with $a<b$ and $c<d$ \emph{cross} if $a<c<b<d$ or $c<a<b<d$.  A chord diagram is \emph{noncrossing} if no two of its chords cross.
\end{definition}
We will draw our partial rooted chord diagrams with $m$ points on a circle, indicating point 1 with an $\times$ and taking the rest in counterclockwise order, and indicating the chords by arcs inside the circle, see \cref{fig partial chord diagram}.  Often rooted chord diagrams are drawn with the points $\{1,2,\ldots, m\}$ from left to right on a line, but we will be making use of a notion of inside that is more visually appealing in the circle representation.

Call a point \emph{bare} if it does not have any chord on it, or equivalently if it is a part of size 1 in the set partition.

\begin{figure}
  \includegraphics{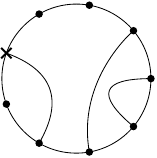}
  \caption{A partial rooted chord diagram.  Here there are three chords and nine points.  Point 1 is marked with an $\times$ and the rest are labelled counterclockwise, so the chords, written as sets of their endpoints as in the definition are $\{1,3\}$, $\{4,7\}$, and $\{5,6\}$.  Points $2$, $8$ and $9$ are bare.  This diagram is noncrossing.}\label{fig partial chord diagram}
\end{figure}

Given a chord in a partial chord diagram, picking one end of the chord to be the \emph{first end} lets us define the \emph{inside} of the chord as the cyclic interval from immediately after the first end until immediately before the other end.  For example, in \cref{fig partial chord diagram} if we take the chord from 4 to 7 as having 4 for its first end, then the inside of that chord includes the points 5 and 6, and so the chord from 5 to 6 is inside the chord from 4 to 7 for this choice of first end.

\begin{definition}\label{def B def}
  Let $\mathcal{B}_{n,m}$ be the set of partial rooted chord diagrams with $n$ chords on $m$ points
  with the following additional information:
  \begin{itemize}
  \item each chord is coloured, blue, red, or black,
  \item each blue and red chord has a designated first end, 
  \end{itemize}
  and with the properties:
  \begin{itemize}
  \item the inside of each red or blue chord has every point in a black chord and no black chords exist which are not inside a red or a blue chord,
  \item the chord diagram is noncrossing.
  \end{itemize}
  Let
  \[
  \mathcal{B} = \bigcup_{\substack{m\geq 2n \\ n\geq 0}} \mathcal{B}_{n,m}.
  \]
\end{definition}

\begin{figure}
  \includegraphics{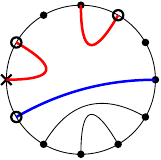}
  \caption{An example of an element of $\mathcal{B}_{5, 12}$.  The red and blue chords are also indicated with a heavier weight, and the first ends are circled.  This is an element of $\mathcal{B}_{5,12}$ because there are 5 chords total, 12 points total, the inside of the blue chord is entirely filled with black chords and in the trivial sense the same is true of the red chords since they have no points on their inside, and there are no other black chords.}\label{fig B object}
\end{figure}

For an example of the definition see \cref{fig B object}.

There are a few things to observe about this somewhat unwieldy definition.  First note that no red or blue chord can be in the inside of another red or blue chord.

Because of this, generically the information of which end of a red or blue chord is not extra information.  If there exists at least one point which is not in any chord, this point must be on the outside of all red and blue chords and so unambiguously defines the first end of each red and blue chord.  Similarly, if there exists a black chord, then this chord determines the inside of the red or blue chord it is contained by and is outside every other red or blue chord, hence defining the first end of each red and blue chord.  Finally, if there is more than one red or blue chord, then since none can be inside another, the dual tree they define must be a star with one leaf for each red or blue chord, and a star with at least two leaves has a well defined centre vertex.  This defines the outside of each red or blue chord and hence their first ends.

Consequently, the only case when the information of the first end cannot be deduced from the rest of the structure is the case when $n=1$ and $m=2$.  None the less, it seems easier to keep this information in the definition since this way the $n=1$ and $m=2$ term does not need tweaking to agree with the physical problem we aim to solve. Additionally, the first end will be a useful notion later, and so given $b\in \mathcal{B}_{n,m}$, define $F(b)\subseteq \{1, 2, \ldots, m\}$ be the set of first ends of $b$ and note that the first ends inherit an order from the order $\{1, 2, \ldots, m\}$.

\begin{lemma}\label{lem gen series}
  The generating function for $\mathcal{B}$ is
  \[
  B(x,y) = \sum_{\substack{m\geq 2n \\ n\geq 0}}\sum_{b\in \mathcal{B}_{n,m}}x^ny^m
   = \frac{\frac{y}{\sqrt{1-4xy^2}} + y^2(1+4x)}{1-y^2(1+4x)}.
  \]
\end{lemma}

\begin{proof}
  This is a fairly standard generating function argument.  For an overview of generating function methods, see for example Part A of \cite{FSbook}.

  Noncrossing complete rooted chord diagrams are well known to be counted by the Catalan numbers, $\mathsf{Cat}_n = \frac{1}{n+1}\binom{2n}{n}$.  Let $C(x) = \sum_{n\geq 0} \mathsf{Cat}_n x^n$ be the generating series for the Catalan numbers, equivalently for noncrossing rooted chord diagrams with $x$ counting the number of chords.  Specifically, saying that $x$ is counting the number of chords means that the coefficient of $x^n$ in $C(x)$ is the number of noncrossing chord diagrams with $n$ chords.

  Then, since every chord has exactly two points, for complete chord diagrams, we can count with two variables, $x$ and $y$, where $x$ counts the number of chords and $y$ counts the number of points simply by substituting $xy^2$ for what was $x$ before.  Namely, $C(xy^2)$ counts noncrossing rooted chord diagrams with $x$ counting number of chords and $y$ counting number of points.  Adding a new outer chord around each noncrossing diagram multiplies the generating function by $xy^2$.  These are precisely the \emph{indecomposable} noncrossing rooted chord diagrams, where here indecomposable means the diagram is nonempty and is not a concatenation of 2 or more nonempty diagrams.  

  Elements of $\mathcal{B}$ are cyclic arrangements of indecomposable noncrossing rooted chord diagrams with the outer chord coloured red or blue (generating function $2xy^2C(xy^2)$) along with bare points (generating function $y$).  If the root is a bare point, then such objects are formed by the root and a sequence of the building blocks described above.  It is a standard fact of generating series that if $F$ is the generating series for some set of objects not including an empty object, then the generating series for sequences of such objects is given by $1/(1-F)$.  Thus the diagrams considered so far have generating function
  \[
    \frac{y}{1-y-2xy^2C(xy^2)}.
    \]
    Here the factor of $y$ in the numerator is for the bare root.  The $y$ is added to $2xy^2C(xy^2)$ in the denominator since we take the sum of the generating series of the two possible (disjoint) types of blocks, and then finally we apply the fact mentioned above about sequences.
    
  On the other hand, if the root is a point in one of the indecomposable pieces then we choose a root in that piece and the rest is a sequence of the building blocks.  Here we need one other standard fact about generating series, the fact that choosing a root in all possible ways corresponds to applying the operator $zd/dz$, where this operator is taken with respect to the variable, in this case $y$, which counts the feature we are choosing among, in this case the points. Thus the diagrams with root in one of the indecomposable pieces has generating function
  \[
    \frac{y\frac{d}{dy}2xy^2C(xy^2)}{1-y-2xy^2C(xy^2)}.
    \]
    Here the numerator gives the generating series for the indecomposable piece which contains the root by the discussion above, while the denominator is the sequence of the rest for the same reason as the earlier calculation.
    
  Adding these two pieces together, using the standard fact that $C(x)=(1-\sqrt{1-4x})/(2x)$, and rearranging, including rationalizing the denominator, we obtain the form of the generating function in the statement. 
\end{proof}

As is common in enumerative combinatorics, write $[x^ny^m]F(x,y)$ for the coefficient of $x^ny^m$ in $F(x,y)$.

\begin{lemma}\label{lem gamma part}
  Let $B(x,y)$ be as in \cref{lem gen series}, then
  \[
    [x^{\lfloor \frac{d}{2} \rfloor + 1}y^{2\lfloor \frac{d}{2}\rfloor +2 + dk}]B(x,y) = \frac{(2dk + 4\lfloor\frac{d}{2}\rfloor + 4)(2dk + 4\lfloor\frac{d}{2}\rfloor)\cdots (2dk+4)}{(\lfloor\frac{d}{2}\rfloor + 1)!}
  \]
  where subsequent factors in the numerator differ by 4, and this is the ratio of gamma factors in the $C_i^{(d)}$ after scaling by $2^{2\lfloor \frac{d}{2}\rfloor + 2}$, namely
  \[
[x^{\lfloor \frac{d}{2} \rfloor + 1}y^{2\lfloor \frac{d}{2}\rfloor +2 + dk}]B(x,y) = \begin{cases}
     \displaystyle 2^{d+2}\frac{\Gamma\left(\frac{d}{2}(k+1) + 2\right)}{\Gamma\left(\frac{d}{2}+2\right)\Gamma\left(1+\frac{dk}{2}\right)} & \text{for even $d$}\\
     \displaystyle 2^{d+1}\frac{\Gamma\left(\frac{d}{2}(k+1)+\frac{3}{2}\right)}{\Gamma\left(\frac{d+3}{2}\right)\Gamma\left(1+\frac{dk}{2}\right)} & \text{for odd $d$}
    \end{cases}.
  \]
\end{lemma}

\begin{proof}
  It is convenient to write out the odd an even $d$ cases separately, so what we want to show is for even $d$
  \[
    [x^{\frac{d}{2}+1}y^{2+d+dk}]B(x,y) =  \frac{(2dk + 2d + 4)(2dk + 2d)\cdots (2dk+4)}{(\frac{d}{2} + 1)!}
     = 2^{d+2}\frac{\Gamma\left(\frac{d}{2}(k+1) + 2\right)}{\Gamma\left(\frac{d}{2}+2\right)\Gamma\left(1+\frac{dk}{2}\right)}
  \]
  and for odd $d$
  \[
    [x^{\frac{d+1}{2}}y^{1+d+dk}]B(x,y) =  \frac{(2dk + 2d + 2)(2dk + 2d-2)\cdots (2dk+4)}{(\frac{d+1}{2})!} 
     =2^{d+1}\frac{\Gamma\left(\frac{d}{2}(k+1)+\frac{3}{2}\right)}{\Gamma\left(\frac{d+3}{2}\right)\Gamma\left(1+\frac{dk}{2}\right)}.
  \]

  For even $d$, the arguments to the gamma functions are all integers so we can directly translate to factorials obtaining
  \begin{align*}
    2^{d+2}\frac{\Gamma\left(\frac{d}{2}(k+1) + 2\right)}{\Gamma\left(\frac{d}{2}+2\right)\Gamma\left(1+\frac{dk}{2}\right)} & = 2^{d+2}\binom{\frac{d}{2}k + \frac{d}{2} + 1}{\frac{d}{2}+1} \\
                                                                                                                             & = 2^{d+2}\frac{(\frac{d}{2}k + \frac{d}{2} + 1)(\frac{d}{2}k+\frac{d}{2}) \cdots (\frac{d}{2}k + 1)}{\frac{d}{2}+1} \\
    & = \frac{(2dk + 2d + 4)(2dk + 2d)\cdots (2dk+4)}{(\frac{d}{2} + 1)!}
  \end{align*}

  For odd $d$ and even $k$, the arguments to the gamma functions are again integers and so similarly we obtain
  \[
    2^{d+1}\frac{\Gamma\left(\frac{d}{2}(k+1)+\frac{3}{2}\right)}{\Gamma\left(\frac{d+3}{2}\right)\Gamma\left(1+\frac{dk}{2}\right)} = 2^{d+1}\binom{\frac{d+1+dk}{2}}{\frac{d+1}{2}}= \frac{(2dk + 2d + 2)(2dk + 2d-2)\cdots (2dk+4)}{(\frac{d+1}{2})!}.
  \]
  When $d$ and $k$ are both odd use the formula for the gamma function at half integers in terms of the double factorial: $\Gamma(\frac{1}{2}+n) = \frac{(2n-1)!!}{2^n}\sqrt{\pi}$.  With this we obtain for $d$ and $k$ both odd
  \begin{align*}
    2^{d+1}\frac{\Gamma\left(\frac{d}{2}(k+1)+\frac{3}{2}\right)}{\Gamma\left(\frac{d+3}{2}\right)\Gamma\left(1+\frac{dk}{2}\right)} & = \frac{2^{\frac{d+1}{2}}(dk+d+1)!!}{(\frac{d+1}{2})! (dk)!!}\\
                                                                                                                                     & = \frac{2^{\frac{d+1}{2}}(dk+d+1)(dk+d-1)\cdots(dk+2)}{(\frac{d+1}{2})!} \\
    & = \frac{(2dk + 2d + 2)(2dk + 2d-2)\cdots (2dk+4)}{(\frac{d+1}{2})!}.
  \end{align*}

  \medskip

  It remains now to do the coefficient extraction on $B(x,y)$.  Note that the only place an odd power of $y$ appears in $B(x,y)$ is in the term $y/\sqrt{1-4xy^2}$ in the numerator.  Consider first the coefficient for even powers of $y$.  For $i>0$, $n\geq 0$ we have
  \[
    [x^ny^{2i}]F(x,y) = [x^ny^{2i}]\frac{y^2(1+4x)}{1-y^2(1+4x)} = [x^n](1+4x)^i = 4^n\binom{i}{n}
  \]
  For $d$ even take $n=\frac{d}{2}+1$, $i=\frac{d+2+dk}{2}$ and so
  \[
    [x^{\frac{d}{2}+1}y^{d+2+dk}]F(x,y) = 4^{\frac{d}{2}+1}\binom{\frac{d+2+dk}{2}}{\frac{d}{2}+1} =  \frac{(2dk + 2d + 4)(2dk + 2d)\cdots (2dk+4)}{(\frac{d}{2} + 1)!}
  \]
  as desired,
  For odd $d$ and even $k$ take $n=\frac{d+1}{2}$, $i=\frac{d+1+dk}{2}$ and so
  \[
     [x^{\frac{d+1}{2}}y^{d+1+dk}]F(x,y) = 4^{\frac{d+1}{2}}\binom{\frac{d+1+dk}{2}}{\frac{d+1}{2}} =  \frac{(2dk + 2d + 2)(2dk + 2d-2)\cdots (2dk+4)}{(\frac{d+1}{2})!}.
   \]
   It remains to consider $d$ and $k$ both odd.
   In this case the relationship between the powers of $x$ and $y$ is important so let $t=xy^2$. Then for $n,i\geq 0$
   \begin{align*}
     [x^ny^{2i+1}]F(x,y) & = [x^ny^{2i}]\frac{1}{\sqrt{1-4xy^2}(1-y^2(1+4x))} \\
                         & = [t^ny^{2j}]\frac{1}{\sqrt{1-4t}(1-(y^2+4t))} \qquad  \text{ where $j=i-n$} \\
                         & = [t^n](1-4t)^{-\frac{1}{2}}\sum_{k\geq j}\binom{k}{j}(4t)^{k-j} \\
                         & = [t^n](1-4t)^{-\frac{1}{2}}\sum_{\ell \geq 0}\binom{\ell+j}{\ell}(4t)^{\ell} \\
                         & = [t^n](1-4t)^{-\frac{1}{2}}(1-4t)^{-j-1} \\
                         & = (-4)^n\binom{-j-\frac{3}{2}}{n} \\
     & = \frac{(4j+6)(4j+10)\cdots (4j+4n+2)}{n!}
   \end{align*}
   For $d$ and $k$ both odd we are interested in the case $n=\frac{d+1}{2}$ and $i = d+dk$ so $j=\frac{dk-1}{2}$, giving
   \[
     [x^{\frac{d+1}{2}}y^{d+1+dk}]F(x,y) = \frac{(2dk + 2d + 2)(2dk + 2d-2)\cdots (2dk+4)}{(\frac{d+1}{2})!}
   \]
   and completing the proof.
\end{proof}

What we have accomplished in this phase is to define some objects, specifically the partial chord diagrams of $\mathcal{B}$, which are counted by the gamma function factor in the expressions for $C_i^{(d)}$ after scaling by $2^{2\lfloor \frac{d}{2}\rfloor+2}$.  Note that we only make use of elements of $\mathcal{B}$ for certain choices of parameters, so it isn't accurate to say that $B(x,y)$ is the generating function for the gamma factors in the $C_i^{(d)}$ with appropriate scaling, since not all elements of $\mathcal{B}$ are used in this interpretation.

\subsection{Phase 2}
For phase 2, we now want an interpretation for the $C_i^{(d)}$.  Unfortunately, there are signs here.  In general subtraction does not have a combinatorial interpretation since it depends on the sets of objects counted by the numbers on either side of the subtraction sign having an appropriate containment relation.  Fortunately in this case, we can give a suitable interpretation where the containment makes sense.

Next note that in odd dimension the $C_i^{(d)}$ are not in general integers, so we will instead give an interpretation for $2^{2\lfloor\frac{d}{2}\rfloor +2}C_i^{(d)}$ and will comment at the end of this section on how we can remove one power of 2 in all cases and in \cref{sec even} on how we can remove all of the powers of two in even dimensions.

This phase is where we'll use the first ends of the red and blue chords, even though that information is redundant in almost all elements of $\mathcal{B}$.  Recall that for $b\in \mathcal{B}$, $F(b)$ is the set of first ends and this set inherits a total order.  Write $F(b)=\{e_1, e_2, \ldots\}$ in this order.
Given a red or blue chord $c$ of some $b\in \mathcal{B}$, say $c$ has \emph{width} $j$ if $c$ has $j-1$ black chords inside it. Write $w(e_j)$ for the width of the red or blue chord with first end $e_j$.  

\begin{definition}
Let $\mathcal{B}_{n,m,c,d}$ be the set of elements of $\mathcal{B}_{n,m}$ with fewer than $c$ consecutive bare points immediately preceding $e_j$ for all $j$ with $\sum_{\ell=1}^{j-1} w(e_\ell) < d$. 
\end{definition}

\Cref{fig B4} gives an example of the definition and the notions of the next two paragraphs.

In the main proof, we'll be inserting bare points before the first ends of red and blue chords.  If $c$ has width $j$ then $c$ has $j$ \emph{insertion places} before its first end.  We will view these $j$ insertion places as distinct, so for a chord of width 2, inserting three bare points into its first insertion place and inserting three bare points into its second insertion place will be viewed as different acts of insertion even though the resulting partial chord diagrams are the same.

With this notion of insertion place in hand, another way to phrase the condition defining $\mathcal{B}_{n,m,c,d}$ is that there are fewer than $c$ consecutive bare points immediately preceding each first end that corresponds to one of the the first $d$ insertion places.

  \begin{figure}
    \includegraphics{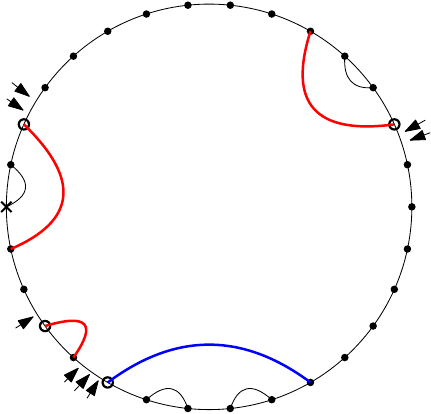}
    \caption{An example of an element of $\mathcal{B}_{8, 30, 2, 3}$.  The insertion places are marked by arrows and the width of each red or blue chords is the number of arrows immediately before its first end.  To begin with, we see that this diagram is an element of $\mathcal{B}_{8,30}$ since it has 30 points, 8 chords and satisfies the required conditions.  Saying further that it is in $\mathcal{B}_{8,30, 2, 3}$ means that the red width 1 chord and the blue chord both have fewer than $2$ free points immediately before them.  It is these two chords that we need to consider because the first three insertion places come before these chords, or equivalently because $w(e_1) = 1 < 3 \leq 4 = w(e_1)+w(e_2)$.}\label{fig B4}
  \end{figure}

\medskip

The $\mathcal{B}_{n,m,c,d}$ with appropriate parameters are the objects which the $2^{2\lfloor\frac{d}{2}\rfloor + 2}C_i^{(d)}$ count.
\begin{thm}\label{thm main thm}
  \[
    2^{2\lfloor\frac{d}{2}\rfloor + 2}C_i^{(d)} =(-1)^{i-1} |\mathcal{B}_{\lfloor\frac{d}{2}\rfloor + 1, 2\lfloor\frac{d}{2}\rfloor + 2 + d(i-1), d, i-1}|.
  \]
\end{thm}

\begin{proof}
  Let's first pull out the sign in the expression for $C_i^{(d)}$, do a trivial rewriting of the binomial coefficients so that they're exactly in the form in which we'll interpret them, and apply \cref{lem gamma part}.
  \begin{align*}
     2^{2\lfloor\frac{d}{2}\rfloor + 2}C_i^{(d)} & = \begin{cases}
    2^{2\lfloor\frac{d}{2}\rfloor + 2}(-1)^{i-1}\displaystyle\sum_{k=0}^{i-1}\binom{i-1}{k}(-1)^{k-i+1} \frac{\Gamma\left(\frac{d}{2}(k+1) + 2\right)}{\Gamma\left(\frac{d}{2}+2\right)\Gamma\left(1+\frac{dk}{2}\right)} & \text{for even $d$}\\
    2^{2\lfloor\frac{d}{2}\rfloor + 2}(-1)^{i-1}\displaystyle\sum_{k=0}^{i-1}\binom{i-1}{k}(-1)^{k-i+1} \frac{\Gamma\left(\frac{d}{2}(k+1)+\frac{3}{2}\right)}{\Gamma\left(\frac{d+3}{2}\right)\Gamma\left(1+\frac{dk}{2}\right)} & \text{for odd $d$}
  \end{cases}\\
    & =(-1)^{i-1}\displaystyle\sum_{k=0}^{i-1}\binom{i-1}{i-1-k}(-1)^{k-i+1} |\mathcal{B}_{\lfloor \frac{d}{2} \rfloor + 1, 2\lfloor \frac{d}{2}\rfloor +2 + dk}|
  \end{align*}
  So what we need to prove is
  \[
    |\mathcal{B}_{\lfloor\frac{d}{2}\rfloor + 1, 2\lfloor\frac{d}{2}\rfloor + 2 + d(i-1), d, i-1}| = \displaystyle\sum_{k=0}^{i-1}\binom{i-1}{i-1-k}(-1)^{k-i+1} |\mathcal{B}_{\lfloor \frac{d}{2} \rfloor + 1, 2\lfloor \frac{d}{2}\rfloor +2 + dk}|
  \]
  To that end, fix $d$ and $i$, and consider an element $b\in \mathcal{B}_{\lfloor \frac{d}{2} \rfloor + 1, 2\lfloor \frac{d}{2}\rfloor +2 + dk}$ for some $0\leq k \leq i-1$. 

  Observe that the sum of the widths of all red and blue chords is just the total number of chords of all three colours, so $b$ has $\lfloor \frac{d}{2} \rfloor + 1$ insertion places total.  Consider the multiset $M_{b,k}$ of partial chord diagrams resulting from choosing $i-1-k$ of the first $i-1$ insertion places of $b$ and inserting $d$ bare points into each of those insertion places.  Each way of inserting adds a total of $d(i-1-k)$ points to $b$ so each element of $M_{b,k}$ is an element of $\mathcal{B}_{\lfloor\frac{d}{2}\rfloor + 1, 2\lfloor\frac{d}{2}\rfloor + 2 + dk + d(i-1-k)} = \mathcal{B}_{\lfloor\frac{d}{2}\rfloor + 1, 2\lfloor\frac{d}{2}\rfloor + 2 + d(i-1)}$.

  Thus we have
  \[
    \displaystyle\sum_{k=0}^{i-1}\binom{i-1}{i-1-k}(-1)^{k-i+1} |\mathcal{B}_{\lfloor \frac{d}{2} \rfloor + 1, 2\lfloor \frac{d}{2}\rfloor +2 + dk}|
    = \displaystyle\sum_{k=0}^{i-1}(-1)^{k-i+1}  |M_{k}|,
  \]
  where
  \[
    M_k =  \biguplus_{b\in\mathcal{B}_{\lfloor \frac{d}{2} \rfloor + 1, 2\lfloor \frac{d}{2}\rfloor +2 + dk}}M_{b,k}
  \]
  is the multiset union of the $M_{b,k}$.  Note that all the $M_k$ are multisets whose elements are elements of the set $\mathcal{B}_{\lfloor\frac{d}{2}\rfloor + 1, 2\lfloor\frac{d}{2}\rfloor + 2 + d(i-1)}$.

  With these insertions all our objects live in the correct set. We only need to check that the signs and multiplicities conspire to cancel all elements of  $\mathcal{B}_{\lfloor\frac{d}{2}\rfloor + 1, 2\lfloor\frac{d}{2}\rfloor + 2 + d(i-1)}$ which are not in $\mathcal{B}_{\lfloor\frac{d}{2}\rfloor + 1, 2\lfloor\frac{d}{2}\rfloor + 2 + d(i-1), d, i-1}$ and to give all elements that are in $\mathcal{B}_{\lfloor\frac{d}{2}\rfloor + 1, 2\lfloor\frac{d}{2}\rfloor + 2 + d(i-1), d, i-1}$ exactly once.

  The second part is simpler to check.  All the insertions were of $d$ bare points into $i-1-k$ of the first $i-1$ insertion places, so the only way to obtain an element of $\mathcal{B}_{\lfloor\frac{d}{2}\rfloor + 1, 2\lfloor\frac{d}{2}\rfloor + 2 + d(i-1), d, i-1}$ is from the $k=i-1$ term (where there is no insertion) and that term is $(-1)^0|\mathcal{B}_{\lfloor\frac{d}{2}\rfloor + 1, 2\lfloor\frac{d}{2}\rfloor + 2 + d(i-1)}|$ which contributes each element of $\mathcal{B}_{\lfloor\frac{d}{2}\rfloor + 1, 2\lfloor\frac{d}{2}\rfloor + 2 + d(i-1), d, i-1}$ exactly once.

  For the first part take $c\in \mathcal{B}_{\lfloor\frac{d}{2}\rfloor + 1, 2\lfloor\frac{d}{2}\rfloor + 2 + d(i-1)}$, $c\not\in\mathcal{B}_{\lfloor\frac{d}{2}\rfloor + 1, 2\lfloor\frac{d}{2}\rfloor + 2 + d(i-1), d, i-1}$.  By definition $c$ has $d$ or more bare points before at least one of its first $i-1$ insertion places, so $c$ appears in at least one $M_{k}$ for some $k<i-1$.  Let $\ell$ be the minimum $k$ for which $c$ appears and take $b\in  \mathcal{B}_{\lfloor \frac{d}{2} \rfloor + 1, 2\lfloor \frac{d}{2}\rfloor +2 + d\ell}$ so that $c\in M_{b,\ell}$.

  Now, we claim that $b$ is the unique $b\in \mathcal{B}_{\lfloor \frac{d}{2} \rfloor + 1, 2\lfloor \frac{d}{2}\rfloor +2 + d\ell}$ so that $c\in M_{b,\ell}$. If not then suppose $b$ and $b'$ were two such partial chord diagrams, then $b$ and $b'$ differ only in their bare points, so build $b''$ so as to agree with $b$ and $b'$ except that whenever $b$ and $b'$ have a different number of bare points before a given red or blue chord, take the minimum of these numbers as the number of bare points before this chord in $b''$.   In whichever of $b$ or $b'$ achieved the minimum at this point, these insertions were all valid, so doing so across $b''$ we get a valid insertion pattern giving $c$ but $b''$ is smaller than $b$ or $b'$ contradicting the minimality of $\ell$ and proving the claim.

  Write the first ends of $b$ in their inherited order, $F(b)=\{e_1, e_2, \ldots\}$ and let $w_j$ be the number of the first $i-1$ insertion places of $b$ that are before $e_j$.  Then for some $t$,  $w_1, w_2, \ldots, w_{t-1}$ are the widths of their respective chords, $0<w_t$ is bounded above by the width of its chord, and $w_{t+1}=w_{t+2} = \cdots = 0$.  Let $i_1, \ldots, i_t$ be the number of insertion places before $e_1, \ldots, e_t$ respectively that must be chosen in order to obtain $c$ from $b$.  Thus the multiplicity of $c$ in $M_{b, \ell}$ is
  \[
    \prod_{j=1}^{t}\binom{w_j}{i_j}.
  \]
  Now for each $1\leq j\leq t$ and any $0\leq i_j'\leq i_j$ we can add $d(i_j-i_j')$ base points to the partial chord diagram immediately before $e_j$ giving a larger partial chord diagram for which $i_j'$ insertion places before $e_j$ still need to be inserted into in order to obtain $c$.  Every such choice is counted separately in the multisets $M_k$ since it comes from a different pair of a chord diagram and a pattern of insertions and every appearance of $c$ in the $M_k$ occurs in this way.
  The total signed contribution of these copies of $c$ is
  \[
    \sum_{\substack{0\leq i_1' \leq i_1 \\\cdots \\ 0 \leq i_t'\leq i_t}}\prod_{j=1}^{t}\binom{w_j}{i_j'}(-1)^{i_1'+i_2'+\cdots+i_t'}
     = \prod_{j=1}^{t} \sum_{0\leq i_j' \leq i_j}\binom{w_j}{i_j'}(-1)^{i_j'} \\
     = 0
  \]
  by the binomial theorem, which completes the proof of the result.

\end{proof}

It would be nice not to have the initial power of $2$ multiplying the $C_i^{(d)}$.  In general, we do need some powers of $2$ since the $C_i^{(d)}$ do sometimes have denominators.  In the next section we'll look at removing the powers of 2 in the even dimensional case, but first observe that since every element of $\mathcal{B}$ has at least one red or blue chord, at least one power of $2$ is easy to remove in all cases.

\begin{cor}
  Let $\widetilde{\mathcal{B}}_{n,m,a,b}$ be $\mathcal{B}_{n,m,a,b}$ where we identify partial chord diagrams that differ only by swapping all red chords for blue chords and vice versa.  Then
  \[
    2^{2\lfloor\frac{d}{2}\rfloor + 1}C_i^{(d)} =(-1)^{i-1} |\widetilde{\mathcal{B}}_{\lfloor\frac{d}{2}\rfloor + 1, 2\lfloor\frac{d}{2}\rfloor + 2 + d(i-1), d, i-1}|.
    \]
\end{cor}

\section{Simplifications in even dimension}\label{sec even}

In even dimension the number of points in the partial chord diagram is even.  In this situation things can be simplified and the power of 2 can be absorbed.  We'll use the following map from elements of $\mathcal{B}$ to binary strings.

\begin{definition}
  Given $b\in \mathcal{B}_{n,2j}$, define $s(b)$ to be the binary string of length $j$ whose $\ell$th entry is $0$ if point $2\ell-1$ is bare in $b$ and $1$ otherwise.
\end{definition}

For example, with $b$ the diagram in \cref{fig B object}, $s(b)$ is $111110$.  The first bit is a $1$ because the root is not a bare point, neither is the third point (which has the first end of a black chord on it).  In fact the only odd-indexed  bare point is the last one, giving the $0$ at the end of $s(b)$.  For another example, see \cref{fig schematic}.

Clearly the map $s$ loses a lot of information about $b$. The next claim is that $s$ is a $4^n$-to-$1$ map, where $n$ is the number of chords, and hence will absorb the power of 2 that we had in the previous section.

\begin{lemma}\label{lem remove the 2s}
  For $b\in \mathcal{B}_{n,2j}$, $s(b)$ has exactly $n$ $1$s.
  
  For every binary string $w$ of length $j$ with $n$ $1$s, $|s^{-1}(w)| = 4^n$.
\end{lemma}

\begin{proof}
  For the first point, note that noncrossing complete rooted chord diagrams have every chord with one even indexed end and one odd indexed end.  Since $b$ is on an even number of points, the parity of the indices of the points is well defined cyclically.  Further $b$ is made of a cyclic sequence of complete noncrossing chord diagrams and bare points and hence the indexing of each complete noncrossing segment is changed by a constant shift.  Thus $b$ also has every chord with one even indexed end and one odd indexed end.  Therefore every chord of $b$ has exactly one end that becomes a $1$ under $s$ and no other $1$s occur in $s(b)$ giving that $s(b)$ has exactly $n$ $1$s.

  For the second part, consider how a block of $1$s can appear in $s(b)$.  Suppose first that there is at least one $0$ in $s(b)$ and consider the blocks in $s(b)$ cyclically.  Take a block of $1$s in $s(b)$.  There is some $0$ before the block, corresponding to a bare point in $b$, then in $b$ there is a sequence of indecomposable noncrossing chord diagrams with the outer chord red or blue and where the indexed points are in the even positions, then another bare point, then a sequence of indecomposable noncrossing chord diagrams with the outer chord red or blue and where the indexed points are in the odd positions, then another bare point, as illustrated in \cref{fig schematic}.  Either sequence of indecomposable diagrams may be empty, but both may not be empty. The first and last bare point may be the same point as we are working cyclically.  Each block of 1s where $s(b)$ contains at least one $0$ appears in this way, and each such configuration in $b$ gives a block of $1$s.

  \begin{figure}
    \includegraphics{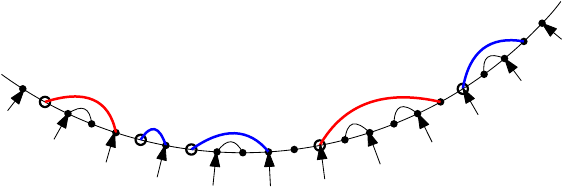}
    \caption{An example of the sequence of indecomposable pieces making up a block of $1$s in $s(b)$.  The arrows indicate the points which directly contribute bits to $s(b)$.  Note that each chord has one end which contributes a bit to $s(b)$. Furthermore that the sequence of indecomposable pieces has exactly one internal bare point (which necessarily is not contributing a bit) separating a sequence of indecomposable pieces where the contributing ends are in the even positions from a sequence of indecomposable pieces where the contributing ends are in the odd positions.  This will always happen, potentially with the sequence on one side of the internal bare point being empty.
    The portion of $s(b)$ given by the part of the chord diagram shown in the figure is $011111111110$ (with ten $1$s in the block of $1$s).}\label{fig schematic}
  \end{figure}
  
  Again using $C(x)$ for the generating function for Catalan numbers, by standard generating function arguments we get that the generating function for a block of $1$s is
  \[
    \left(\frac{1}{1-2xC(x)}\right)\left(\frac{1}{1-2xC(x)}\right) -1 =
      \frac{1}{\sqrt{1-4x}}\frac{1}{\sqrt{1-4x}}-1 = \frac{4x}{1-4x}.
    \]
    Extracting coefficients we get that every block of $t$ $1$s occurs $4^t$ times, so arguing similarly for each block of $1$s, we get that each binary string $w$ with at least one $0$ occurs $4^n$ times, as desired.

    It remains to consider the case that $s(b)$ has no $0$s.  In this case we simply have a cyclic sequence of indecomposable noncrossing chord diagrams with the outer chord coloured red or blue.  As in \cref{lem gen series} we can obtain the generating series for these, by choosing a point on one indecomposable diagram as the root and then taking a sequence of others.  Note that since $x$ counts the chords and each chord has two points, choosing a point is done by the operator $2x\frac{d}{dx}$ in this case.  This gives the generating function
    \[
      \frac{2x\frac{d}{dx}2xC(x)}{1-2xC(x)} = \frac{4x}{\sqrt{1-4x}}\frac{1}{\sqrt{1-4x}} = \frac{4x}{1-4x}
    \]
    giving this binary string also occurs $4^n$ times and completing the proof.
\end{proof}

The previous lemma tells us that in the even dimensional case we can directly reinterpret the phase 1 numbers, that is the ratio of gamma functions appearing in the $C_i^{(d)}$ for even $d$, as the number of binary strings of length $d/2+1+dk/2$ and with $d/2+1$ $1$s.  We could also see that fact more directly by writing the ratio of gamma functions as a binomial coefficient in this case, but the lemma gives us the direct connection between the phase 1 construction and these strings, hence allowing us to carry over everything additional we know from the phase 1 construction.

From here we can redo the phase 2 argument adding $0$s and checking the cancellation occurs.  There is one important detail.  Since we no longer know how each block of $1$s came from a sequence of indecomposable noncrossing chord diagrams, we no longer know where the first ends of the red or blue chords were.  However, we can now simply take an insertion place before every $1$.  The same argument for the cancellation holds, but now is even easier since we no longer need to worry about multiplicities from insertions into different places giving the same result.  Taking this all together we obtain the following theorem.

\begin{thm}
  For $d$ even $(-1)^{i-1}C_i^{(d)}$ is the number of binary strings of length $\frac{d}{2}+1+\frac{d(i-1)}{2}$ with $\frac{d}{2}+1$ 1s and where there are fewer than $\frac{d}{2}$ $0$s immediately before any of the first $i-1$ $1$s.
\end{thm}

\begin{proof}
  Use \cref{lem remove the 2s} and then argue as the the proof of \cref{thm main thm} with insertion places before the first $i-1$ $1$s and with insertions of $\frac{d}{2}$ $0$s.
\end{proof}

These objects are arguably mildly more natural if we interpret the binary strings as lattice paths starting at the origin and where $0$ becomes a step one unit up and $1$ becomes a step one unit right.  Then rephrasing we obtain the following corollary.

\begin{cor}
  For $d$ even $(-1)^{i-1}C_i^{(d)}$ is the number of lattice walks from $(0,0)$ to $(\frac{d}{2}+1, \frac{d(i-1)}{2})$ and using steps $\rightarrow$ and $\uparrow$ with the additional property that there are fewer than $\frac{d}{2}$ up steps with $x$-coordinate $j$ for $0\leq j \leq i-2$.
\end{cor}

Many other interpretations would be similarly possible.

\section{Discussion}

We gave a combinatorial interpretation that captures the cancellations in the expression for the coefficients of the causal set d'Alembertian given in \cite{Gdalembertian}.  While this does not obviously yield a closed form expression, hopefully the reader finds this more satisfying hence in some sense an answer to Glaser's request for a more aesthetically satisfying expression.  Combinatorial interpretations can be potentially useful even when they don't give nice formulas, as they can show structure that may not be apparent otherwise and are well suited to proving identities or properties of the counts by combinatorial means.

From a combinatorial perspective it is tempting to rewrite the formulas for $B^{(d)}$ by pulling out a factor of $\beta_d$ giving
\[
B^{(d)}\phi(x) = \frac{\beta_d}{\ell^2} \left(\frac{\alpha_d}{\beta_d} \phi(x) + \sum_{i=1}^{\lfloor \frac{d}{2} \rfloor + 2 }c_i^{(d)}\sum_{y\in L_i}\phi(y)\right).
\]
We know \cite{Gdalembertian} that
\begin{align*}
  \frac{\alpha_d}{\beta_d} & = \begin{cases} -\frac{\Gamma(d)}{\Gamma(\frac{d}{2}+2)\Gamma(\frac{d}{2})} & \text{for even $d$} \\
    -\frac{2^{d-1}}{d+1} & \text{for odd $d$} \end{cases} \\
  & = \begin{cases} -\frac{1}{2}\mathsf{Cat}_{\frac{d}{2}} & \text{for even $d$} \\
    -\frac{2^{d-1}}{d+1} & \text{for odd $d$} \end{cases}
\end{align*}
where $\mathsf{Cat}_n$ is the $n$th Catalan number.
This suggests some $i=0$ version of our objects which are counted by Catalan numbers in the even dimensional case and by $\frac{1}{d+1}$ in the odd dimensional case, with scaling by powers of $2$ as necessary.    It is not clear how to fit this in to the picture described above.  Likely an alternate interpretation would be needed.

Finally, one may ask what physical use these interpretations are.  While they do not seem to be immediately so useful, they do have some physically nice properties.  First of all, by the very nature of combinatorial interpretations, they explain the cancellations between the terms in the original alternating sum form of the coefficients.  Second, these interpretations are unified in form across both even and odd dimension.  This is physically appealing because ultimately one would like to have a way of understanding a causal set action without knowing the dimension in advance, and so as a first step a physically useful interpretation should at least not take an entirely different form across dimensions of different parities.  Furthermore, while the dimension does figure into these interpretations in an important way, determining the number of chords, number of points, and constraining the number of bare points before certain first ends, at least we are counting the same kind of thing in each dimension, with only these parameters changing with the dimension, so how the behaviour changes with $d$ is reasonably well controlled.

Since these are the coefficients of the number of intervals of size $i$, what would be even better than the interpretations developed here, would be some way to reinterpret our particular classes of chord diagrams so that they count the number of ways to do some specific thing on the intervals, or so that there's some way to decorate the causal set so that each interval is decorated the number of ways given by these chord diagrams.  In this way, the d'Alembertian would become simply an alternating sum, without coefficients, of the number of such things on each interval or the number of such decorated intervals, and hence, in a more profound sense, would be a natural and intrinsic property of the causal set or decorated causal set, without needing to rely on magic coefficients.

\bibliographystyle{plain}
\bibliography{main}

\begin{thebibliography}{10}

\bibitem{ASSdalembertian}
Siavash Aslanbeigi, Mehdi Saravani, and Rafael~D. Sorkin.
\newblock Generalized causal set d’{A}lembertians.
\newblock {\em JHEP}, 2014:24, 2014.

\bibitem{Bdalembertian}
Alessio Belenchia.
\newblock Universal behavior of generalized causal set d’{A}lembertians in
  curved spacetime.
\newblock {\em Classical and Quantum Gravity}, 33(13):135011, 2016.

\bibitem{BBDcntlimit}
Alessio Belenchia, Dionigi M.~T. Benincasa, and Fay Dowker.
\newblock The continuum limit of a 4-dimensional causal set scalar
  d’{A}lembertian.
\newblock {\em Classical and Quantum Gravity}, 33(24):245018, 2016.

\bibitem{BDdalembertian}
Dionigi M.~T. Benincasa and Fay Dowker.
\newblock Scalar curvature of a causal set.
\newblock {\em Phys. Rev. Lett.}, 104:181301, 2010.

\bibitem{CCSsuppress}
P.~Carlip, S.~Carlip, and S.~Surya.
\newblock Path integral suppression of badly behaved causal sets.
\newblock {\em Classical and Quantum Gravity}, 40(9):095004, 2023.

\bibitem{CCSsuppress2}
Peter Carlip, Steve Carlip, and Sumati Surya.
\newblock The {E}instein–{H}ilbert action for entropically dominant causal
  sets.
\newblock {\em Classical and Quantum Gravity}, 41(14):145005, 2024.

\bibitem{DNYlambda}
Santanu Das, Arad Nasiri, and Yasaman~K. Yazdi.
\newblock Aspects of everpresent $\lambda$. part {I}. a fluctuating
  cosmological constant from spacetime discreteness.
\newblock {\em Journal of Cosmology and Astroparticle Physics}, 2023(10):047,
  2023.

\bibitem{dBEP}
Gustavo~P. de~Brito, Astrid Eichhorn, and Christopher Pfeiffer.
\newblock Higher-order curvature operators in causal set quantum gravity.
\newblock {\em The European Physical Journal Plus}, 138:592, 2023.

\bibitem{Dboundary}
Fay Dowker.
\newblock Boundary contributions in the causal set action.
\newblock {\em Classical and Quantum Gravity}, 38(7):075018, 2021.

\bibitem{DGdalembertian}
Fay Dowker and L.~Glaser.
\newblock Causal set d'{A}lembertians for various dimensions.
\newblock {\em Classical and Quantum Gravity}, 30(19):195016, 2013.

\bibitem{DLLJboundary}
Fay Dowker, Roger Liu, and Daniel Lloyd-Jones.
\newblock Timelike boundary and corner terms in the causal set action, 2024.
\newblock arXiv:2501.00139.

\bibitem{DSreview}
Fay Dowker and Sumati Surya.
\newblock {\em The Causal Set Approach to the Problem of Quantum Gravity},
  pages 1--14.
\newblock Springer Nature Singapore, Singapore, 2023.

\bibitem{FSbook}
Philippe Flajolet and Robert Sedgwick.
\newblock {\em Analytic Combinatorics}.
\newblock Cambridge, 2009.

\bibitem{Gdalembertian}
L.~Glaser.
\newblock A closed form expression for the causal set d'{A}lembertian.
\newblock {\em Classical and Quantum Gravity}, 31(9):095007, 2014.

\bibitem{Gcstmc}
L.~Glaser.
\newblock {\em Computer Simulations of Causal Sets}, pages 1--24.
\newblock Springer Nature Singapore, Singapore, 2023.

\bibitem{GDFcstaqft}
Jonathan Gorard and Julia Dannemann-Freitag.
\newblock Axiomatic quantum field theory in discrete spacetime via multiway
  causal structure: The case of entanglement entropies, 2023.
\newblock arXiv:2301.12455.

\bibitem{Jdiscreteness}
Steven Johnston.
\newblock Correction terms for propagators and d’{A}lembertians due to
  spacetime discreteness.
\newblock {\em Classical and Quantum Gravity}, 32(19):195020, 2015.

\bibitem{BLMS}
Bombelli L, Lee J, Meyer D, and Sorkin R.
\newblock Space-time as a causal set.
\newblock {\em Phys Rev Lett}, 59:521--524, 1987.

\bibitem{LCsuppress}
S.~P. Loomis and S.~Carlip.
\newblock Suppression of non-manifold-like sets in the causal set path
  integral.
\newblock {\em Classical and Quantum Gravity}, 35(2):024002, 2017.

\bibitem{MWboundary}
Ludovico Machet and Jinzhao Wang.
\newblock On the continuum limit of {B}enincasa–{D}owker–{G}laser causal
  set action.
\newblock {\em Classical and Quantum Gravity}, 38(1):015010, 2020.

\bibitem{MSSsuppress}
Abhishek Mathur, Anup~Anand Singh, and Sumati Surya.
\newblock Entropy and the link action in the causal set path-sum.
\newblock {\em Classical and Quantum Gravity}, 38(4):045017, 2020.

\bibitem{MYZlambda}
Heidar Moradi, Yasaman~K. Yazdi, and Miguel Zilhão.
\newblock Fluctuations and correlations in causal set theory.
\newblock {\em Classical and Quantum Gravity}, 42(4):045017, 2025.

\bibitem{Spdcst}
Samuel Shuman.
\newblock Partial derivatives on causal sets, 2024.
\newblock arXiv:2405.10574.

\bibitem{Sdalembertian}
Rafael~D. Sorkin.
\newblock Does locality fail at intermediate length scales?
\newblock In Daniele Oriti, editor, {\em Approaches to Quantum Gravity: Toward
  a New Understanding of Space, Time and Matter}, page 26–43. Cambridge
  University Press, 2009.

\bibitem{Scstreview}
Sumati Surya.
\newblock The causal set approach to quantum gravity.
\newblock arXiv:1903.11544.

\bibitem{Trotter}
William~T. Trotter.
\newblock {\em Combinatorics and Partially Ordered Sets: Dimension Theory}.
\newblock Johns Hopkins University Press, 1992.

\bibitem{Yeverything}
Yasaman~K. Yazdi.
\newblock Everything you always wanted to know about how causal set theory can
  help with open questions in cosmology, but were afraid to ask.
\newblock {\em Modern Physics Letters A}, 39(01):2330003, 2024.

\end{thebibliography}

\end{document}